\documentclass{amsart}
\usepackage[utf8]{inputenc}

\usepackage{amsmath,amssymb,amsthm}
\usepackage{todonotes}
\usepackage{david}

\newtheorem{theorem}[subsection]{Theorem}
\newtheorem{proposition}[subsection]{Proposition}
\newtheorem{lemma}[subsection]{Lemma}
\newtheorem{corollary}[subsection]{Corollary}
\newtheorem{sublemma}[subsection]{Sublemma}
\theoremstyle{definition}
\newtheorem{conjecture}[subsection]{Conjecture}
\newtheorem{remark}[subsection]{Remark}
\newtheorem{fact}[subsection]{Fact}

\newcommand{\bN}{\mathbb N}
\newcommand{\bZ}{\mathbb Z}
\newcommand{\bF}{\mathbb F}
\newcommand{\bQ}{\mathbb Q}
\newcommand{\bQbar}{\overline\bQ}
\newcommand{\bR}{\mathbb R}
\newcommand{\bC}{\mathbb C}

\newcommand{\X}{\mathcal X}
\newcommand{\Y}{\mathcal Y}
\newcommand{\D}{\mathcal D}
\newcommand{\cO}{\mathcal O}
\newcommand{\cU}{\mathcal U}

\newcommand{\glob}{\mathrm{glob}}
\newcommand{\llbrack}{[\![}
\newcommand{\rrbrack}{]\!]}

\newcommand{\W}{\mathrm W}
\newcommand{\Ubar}{\overline U}

\newcommand{\sD}{\mathsf D}
\newcommand{\sd}{\mathsf d}

\newcommand{\REP}{\mathbf{Rep}}
\newcommand{\VEC}{\mathbf{Vec}}

\newcommand{\NS}{\mathrm{NS}}

\title[Bounds on the Chabauty--Kim Locus]{Bounds on the Chabauty--Kim Locus of Hyperbolic Curves}
\author[L.\ Alexander Betts, David Corwin and Marius Leonhardt]{L.\ Alexander Betts, David Corwin and Marius Leonhardt}
\date{\today}

\begin{document}
	
	\maketitle
	
	\begin{abstract}
		Conditionally on the Tate--Shafarevich and Bloch--Kato Conjectures, we give an explicit upper bound on the size of the $p$-adic Chabauty--Kim locus, and hence on the number of rational points, of a smooth projective curve $X/\bQ$ of genus~$g\geq2$ in terms of~$p$,~$g$, the Mordell--Weil rank~$r$ of its Jacobian, and the reduction types of~$X$ at bad primes. This is achieved using the effective Chabauty--Kim method, generalising bounds found by Coleman and Balakrishnan--Dogra using the abelian and quadratic Chabauty methods.
	\end{abstract}
	
	\section{Introduction}
	
	Let~$X/\Qb$ be a smooth proper connected curve of genus $g \ge 2$ and let~$p$ be a prime of good reduction. The Chabauty--Kim method aims to study the rational points on~$X$ through the \emph{Chabauty--Kim locus} $X(\bQ_p)_\infty$, which is a subset of the $\bQ_p$-points of~$X$ containing the rational points \cite{kim09}. Assuming the Bloch--Kato Conjecture, Kim proved that the Chabauty--Kim locus~$X(\bQ_p)_\infty$ is finite for every such curve~$X$, and hence verified that the set of rational points is also finite \cite[\S3]{kim09}. In this paper, we establish an effective version of Kim's result, showing that, under the same assumptions, one obtains an explicit bound on the size of the Chabauty--Kim locus~$X(\bQ_p)_\infty$ in terms of natural arithmetic invariants associated to~$X$.

	
	
	
	\begin{theorem}\label{thm:main}
		Assume the Tate--Shafarevich\footnote{One may alternatively eliminate use of the Tate--Shafarevich Conjecture by replacing $r$ with the $p^\infty$-Selmer rank of the Jacobian.} and Bloch--Kato\footnote{More specifically, we need only Conjecture \ref{conj:BK}. In fact, we need it only for $V$ appearing in $H^1_{\et}(X_{\overline{\Qb}};\Qp)^{\otimes n} (1)$ for $n \in \Nb$.} Conjectures. Then the Chabauty--Kim locus~$X(\bQ_p)_{\infty}$ is finite, and its size is bounded by
		\[
		\#X(\bQ_p)_{\infty} \leq \kappa_p\cdot(p+1+2g\sqrt{p})\cdot\prod_{\ell}n_\ell\cdot(4g-2)^{2^{2r+4}}\cdot(2g)^{2^{4r+7}} \,,
		\]
		where
		$n_\ell$ is the number of irreducible components on the mod-$\ell$ special fibre of the minimal regular model of~$X$, and
		\[\kappa_p \coloneqq \begin{cases} 1 + \frac{p-1}{(p-2)\log(p)} &\mbox{if } p \neq 2, \\
			2 + \frac{2}{\log(2)} & \mbox{if } p=2. \end{cases}
		\]
	\end{theorem}
	
	Since~$X(\bQ)\subseteq X(\bQ_p)_\infty$ for all~$p$, Theorem~\ref{thm:main} gives, \emph{a posteriori}, a bound on the size of~$X(\bQ)$.
	
	\begin{corollary}\label{cor:main}
		Assume the Tate--Shafarevich and Bloch--Kato Conjectures. Then the size of $X(\bQ)$ is bounded by
		\[
		\#X(\bQ) \leq \kappa_p\cdot(p+1+2g\sqrt{p})\cdot\prod_{\ell}n_\ell\cdot(4g-2)^{2^{2r+4}}\cdot(2g)^{2^{4r+7}} \,.
		\]
	\end{corollary}
	
	
	

	The idea of using Chabauty-like methods to bound rational points dates back to Coleman \cite{ColemanChabauty}, who used abelian Chabauty to bound the number of rational points on curves of small Mordell--Weil rank. Coleman's ideas were further refined in \cite{StollUniform,KRZB18} to allow for bad reduction, and generalised using quadratic Chabauty in \cite{BalaDograEffective} to allow slightly higher rank. Our results depend on standard conjectures but hold for any higher genus curve over $\Qb$. A general method for using Chabauty--Kim to bound numbers of rational points was set up in \cite{alex:effective}, and Theorem~\ref{thm:main} unpacks what this method implies under these standard conjectures.
	
	\begin{remark}
		As the bound in Corollary \ref{cor:main} is completely explicit, it gives one in theory a strategy to try to disprove the Bloch--Kato Conjecture, simply by exhibiting a curve~$X$ with more rational points than the corollary predicts. 
		
	\end{remark}
	
	\begin{remark}\label{rmk:BK_using_Iwasawa}
		While the Bloch--Kato Conjecture is open in general, it is possible to verify specific cases using Iwasawa theory by computing $p$-adic $L$-functions. More specifically, an Iwasawa Main Conjecture for $V$ relates vanishing of the Selmer group of $V$ to non-vanishing of values of the $p$-adic $L$-function of $V$, which can be computationally verified. The case of $V=\HH^1_{\et}(E_{\bQbar};\Qp)$ for a non-CM elliptic curve $E$ is discussed in \cite[\S 2.4]{CorwinMECK}. More generally, there has been much progress in the theory of Euler systems in recent years (e.g., in work of Loeffler and Zerbes), which in turn leads to proofs of more cases of the Iwasawa Main Conjecture.
	\end{remark}
	
	\begin{remark}
		
		
		Although our main interest lies in bounding the size of the Chabauty--Kim locus in Theorem~\ref{thm:main} (e.g., in the context of Kim's Conjecture \cite[Conjecture 3.1]{nabsd}), we make some remarks about how the resulting bound on rational points compares to other bounds in the literature.
		
		
		R\'emond \cite[Th\'eor\`eme 1]{Remond11} proves unconditionally that $\#X(\Qb) \le M^{2^{3^{D^2}}}$ for $X$ the zero set of $f(x,y) \in \Zb[x,y]$ of degree $D$ with coefficients of absolute value less than $M$. Although conditional, our bound is favorable in that it 1) applies to the Chabauty--Kim locus and therefore has consequences for Kim's Conjecture \cite[Conjecture 3.1]{nabsd}, 2) is polynomial rather than triply exponential in the genus, and 3) depends only on $g$, $r$, and reduction data.
		
		In this last sense, our result is a partial uniformity statement for the Chabauty--Kim locus. It is therefore a step towards the Uniformity Conjecture \cite[Conjecture 1.1]{KRZB18}\footnote{A ``uniformity conjecture'' is found much earlier in \cite{CHM97}, although that version depends on the number field $K$ rather than merely its degree. The version stated above follows from the Bombieri--Lang Conjecture by \cite[Corollary 1.2]{pacelliUniform}.}, which states that for any integers $d\geq1$ and $g\geq2$ there should exist an integer~$N(g,d)$ such that for any smooth projective curve~$X$ of genus~$g$ over a number field~$K$ of degree $[K:\bQ]\leq d$, we have
		\[
		\#X(K)\leq N(g,d) \,.
		\]
		
		The best known results in the direction of this conjecture are those of Dimitrov--Gao--Habegger \cite[Theorem~1.1]{dimitrov2021uniformity}, who prove that for $g,d$ as above, there is a constant~$c(g,d)$ for which
		\[
		\# X(K)\leq c(g,d)^{r+1} \,.
		\]
		The dependence of the constant~$c(g,d)$ on~$d$ has been removed by K\"uhne \cite{kuehne2021equidistribution}.
		
		On the one hand, Corollary~\ref{cor:main} compares unfavourably with that of Dimitrov--Gao--Habegger, in that our bound works only over $K=\bQ$, requires assuming some standard conjectures, depends on more than just~$r$ and~$g$, and grows considerably faster in~$r$. On the other hand, our bound is completely explicit, and allows us, for example, to understand how our bound varies if we fix~$r$ and the reduction data and vary~$g$ (our bound is polynomial in~$g$, of very large degree). We understand, from email discussions with the authors of \cite{dimitrov2021uniformity} and \cite{kuehne2021equidistribution}, that it would require a significant amount of extra work to make their bounds more explicit, even if possible in principle.
	\end{remark}
	
	\subsection{S-integral points}
	
	Our result also applies in a modified form to affine hyperbolic curves. Fix a finite set~$S$ of rational primes of size $s=\#S$, and let~$D$ be a divisor in $X$ of degree~$n$. We drop the requirement that $g \ge 2$, replacing it with the condition $2g+n >2$. Let~$\D$ be the closure of~$D$ in the minimal regular model~$\X$ of~$X$, and let $\Y\coloneqq \X\setminus\D$. Assume that $\Y(\bZ_S)\neq\emptyset$, and fix a prime~$p\notin S$ of good reduction for~$(\X,\D)$, i.e.\ such that $\X_{\bF_p}$ is smooth and $\D_{\bF_p}$ is \'etale.
	
	The Chabauty--Kim locus\footnote{There are several different definitions of the Chabauty--Kim locus in the literature, depending on exactly which local conditions one wants to impose. In this paper we will always mean the Chabauty--Kim locus as defined in \cite[\S1.2.1]{BettsDogra20}; in the case that~$Y$ is projective, this agrees with the definition in \cite[\S2.7]{nabsd}.} is denoted $\Y(\bZ_p)_{S,\infty}$. We let $n_{\ell}$ denote the number of irreducible components on the mod-$\ell$ special fibre of the chosen model~$\X$. We also let $\rho\coloneqq \rank\NS(\bQ)$ denote the rank of the rational N\'eron--Severi group of the Jacobian of $X$, and set~$\bar s=\max\{s+1-\rho-\#|D|,0\}$ with~$|D|$ the set of closed points of~$D$.
	
	\begin{theorem}\label{thm:mainS}
		Assume the Tate--Shafarevich and Bloch--Kato Conjectures. Then the Chabauty--Kim locus~$\Y(\bZ_p)_{S,\infty}$ is finite, and its size is bounded by
		\[
		\#\Y(\bZ_p)_{S,\infty} \leq \kappa_p\cdot\#\Y(\bF_p)\cdot\prod_{\ell\notin S}n_\ell\cdot\prod_{\ell\in S}(n_\ell+n)\cdot(4g+2n-2)^M\cdot(2g+n)^{(M^2-M)/2} \,,
		\]
		where~$M=2^{2r+2\bar s+4}$.
	\end{theorem}
	
	The inclusion $\Y(\bZ_S)\subset\Y(\bZ_p)_{S,\infty}$ shows that we have the same bound for the number of $S$-integral points of $\Y$.
	We also observe that almost all factors of the bound depend only on $Y=X\setminus D$, only the factors $n_\ell$ depend on the integral model $\Y$.
	
	\begin{remark}
		Theorem~\ref{thm:mainS} holds both for affine and projective hyperbolic curves; the projective case implies Theorem~\ref{thm:main} by the Hasse--Weil bound.
	\end{remark}
	
	\begin{remark}
		Theorem \ref{thm:mainS} is really of interest only in the case $g\leq1$. If~$Y$ is an affine curve of genus~$\geq2$, with smooth completion~$X$, then typically the upper bound on $\# X(\bQ_p)_{\infty}$ from Theorem~\ref{thm:main} is smaller than the upper bound on $\#\Y(\bZ_p)_{S,\infty}$ from Theorem~\ref{thm:mainS}. Since~$\Y(\bZ_p)_{S,\infty}$ is a subset of~$X(\bQ_p)_\infty$, we see that Theorem~\ref{thm:mainS} is typically not telling us anything new in this case.
	\end{remark}
	
	We briefly discuss the proof of Theorem~\ref{thm:main}. Let~$\Ubar$ denote a Galois-equivariant quotient of the $\bQ_p$-pro-unipotent \'etale fundamental group~$U$ of~$Y_{\bQbar}$ based at a point~$b\in Y(\bZ_S)$, where~$p$ is a prime of good reduction as usual. One often takes~$\Ubar$ finite-dimensional, but this is not necessary for our purposes. Associated to~$\Ubar$ is a (refined) Chabauty--Kim locus
	\[
	\Y(\bZ_p)_{S,\Ubar}\subseteq\Y(\bZ_p)
	\]
	containing the $S$-integral points, the locus $\Y(\bZ_p)_{S,\infty}$ of earlier corresponding to the case that~$\Ubar$ is the whole fundamental group. We consider the two power series
	\begin{align*}
		\HS_\loc(t) &\coloneqq  \prod_{k\geq1}(1-t^k)^{-\dim\HH^1_f(G_p;V_k)} \,, \\
		\HS_\glob(t) &= (1-t^2)^{-s}\prod_{k\geq1}(1-t^k)^{-\dim\HH^1_f(G_\bQ;V_k)} \,,
	\end{align*}
	where~$V_k\coloneqq \gr^\W_{-k}\Ubar$ denotes the $k$th graded piece of the weight filtration on~$\Ubar$ and $\HH^1_f(G_p;-)$ and $\HH^1_f(G_\bQ;-)$ denote the local and global Bloch--Kato Selmer groups, respectively. Writing~$c_i^\loc$ and~$c_i^\glob$ for the coefficients of~$\HS_\loc$ and~$\HS_\glob$, the effective Chabauty--Kim Theorem of \cite[Theorem~B]{alex:effective} provides an explicit upper bound on the size of $\Y(\bZ_p)_{S,\Ubar}$ in terms of the least integer~$m$ such that the inequality
	\begin{equation}\label{eq:coeff_ineq}\tag{$\dagger$}
		\sum_{i=0}^mc_i^\glob < \sum_{i=0}^m c_i^\loc
	\end{equation}
	holds. What we accomplish in this note is to compute both $\HS_\loc(t)$ and $\HS_\glob(t)$ in the case that~$\Ubar$ is the whole fundamental group, the latter depending on the Tate--Shafarevich and Bloch--Kato Conjectures, and use this to show that~\eqref{eq:coeff_ineq} holds for some~$m\leq 2^{2r+2\bar s+4}$. Theorems~\ref{thm:main} and~\ref{thm:mainS} then follows from \cite[Theorems~B~\&~6.2.1]{alex:effective}.
	
	\subsection{Asymptotic bounds for once-punctured CM elliptic curves}
	
	In the above context, it is natural to consider how the bounds on the size of $\Y(\bZ_p)_{S,\Ubar}$ depend on the chosen quotient~$\Ubar$. We will show in an appendix that the answer to this question can be quite subtle. Specifically, we will examine the case of a once-punctured CM elliptic curve~$Y$ and consider a particular metabelian quotient~$U'$ of~$U$ studied in~\cite{kimCM}.
	Writing $\Y_{\min}$ for the complement of the zero section in the minimal regular model of $E$, we use the quotient~$U'$ to bound the Chabauty--Kim locus of ~$\Y_{\min}$, obtaining the following.
	
	\begin{theorem}\label{thm:mainPL}
		Assume the Tate--Shafarevich Conjecture and assume that Conjecture~\ref{conj:non-vani-padicLfct} on the non-vanishing of certain $p$-adic $L$-functions holds for~$E$. Then there is an absolute constant~$C$ and a constant~$\kappa=\kappa(E,p)$ depending on~$p$ and~$E$ such that
		\[
		\#\Y_{\min}(\bZ_p)_{S,U'} \leq \kappa\cdot\exp\bigl(C\cdot (r+s)^3\log(r+s)^3\bigr)
		\]
		whenever~$r+s\geq2$.
	\end{theorem}
	
	Without assuming any conjectures, we still obtain a very similar result.
	
	\begin{theorem}\label{thm:mainPL2}
		There is an absolute constant~$C$ and constants $r_1=r_1(E,p)\geq r$, $\kappa=\kappa(E,p)$ such that we have
		\[
		\#\Y_{\min}(\bZ_p)_{S,U'} \leq \kappa\cdot\exp(C\cdot (r_1+s)^3\log(r_1+s)^3)
		\]
		whenever~$r_1+s\geq2$.
	\end{theorem}
	
	What is surprising about Theorem~\ref{thm:mainPL2} is that for~$s\gg0$ it actually gives a \emph{better} upper bound on the size of~$\Y_{\min}(\bZ_p)_{S,\infty}$ than the one coming from Theorem~\ref{thm:mainS}, even though we have the containment
	\[
	\Y_{\min}(\bZ_p)_{S,\infty}\subseteq\Y_{\min}(\bZ_p)_{S,U'}
	\]
	the other way around. Indeed, for fixed~$Y$ and $p$, the bound on $\#\Y_{\min}(\bZ_p)_{S,\infty}$ from Theorem~\ref{thm:mainS} on $\# \Y_{\min}(\bZ_p)_{S,U'}$ is doubly exponential in~$s$, while the bound from Theorem~\ref{thm:mainPL2} is exponential-of-polynomial in~$s$.
	
	We remark that this phenomenon has geometric meaning. In the effective Chabauty--Kim method, as formulated in \cite{alex:effective}, one studies the localisation map
	\[
	\loc_p\colon \Sel_{S,\Ubar} \to \HH^1_f(G_p;\Ubar)
	\]
	from the global to the local Selmer variety attached to the quotient~$\Ubar$. The affine rings of both Selmer varieties carry a natural weight filtration, and
	\[
	\loc^*_p\colon \cO(\HH^1_f(G_p;\Ubar)) \to \cO(\Sel_{S,\Ubar})
	\]
	preserves weight filtrations. Any non-zero element~$f$ of the kernel of $\loc^*_p$ gives rise to a non-zero Coleman analytic function on~$\Y_{\min}(\bZ_p)$ vanishing on~$\Y_{\min}(\bZ_p)_{S,\Ubar}$, and the number of zeroes of this Coleman function can be bounded explicitly in terms of the weight of~$f$. The way one produces such elements of~$\ker(\loc^*_p)$ in \cite{alex:effective} is by a simple dimension argument: when inequality~\eqref{eq:coeff_ineq} holds, then one has $\dim W_m\cO(\Sel_{S,\Ubar}) < \dim W_m\cO(\HH^1_f(G_p;\Ubar))$, and hence~$\loc^*_p$ must have non-trivial kernel in weight~$\leq m$.
	
	What Theorem~\ref{thm:mainPL} shows is that the map $\loc^*_p$ can have non-trivial kernel even when inequality~\eqref{eq:coeff_ineq} fails. Indeed, for~$s\gg0$ we can choose a value of~$m$ such that~\eqref{eq:coeff_ineq} holds for~$\Ubar=U'$ but not for~$\Ubar=U$. Since $\cO(\HH^1_f(G_p;U'))\subseteq\cO(\HH^1_f(G_p;U))$, we know that~$\loc^*_p\colon W_m\cO(\HH^1_f(G_p;U)) \to W_m\cO(\Sel_{S,U})$ has non-trivial kernel, even though its codomain has larger dimension than its domain.
	
	A version of this phenomenon was already seen in~\cite{KimTangential} over totally real number fields, where examples were given where the localisation map has non-dense image even when $\dim \HH^1_f(G_p;\Ubar) \leq \dim \Sel_{S,\Ubar} < \infty$.
	
	\subsection{Acknowledgements}
	
	L.A.B.\ is supported by the Simons Collaboration on Arithmetic Geometry, Number Theory, and
	Computation under grant number 550031. This work began while D.C.\ was supported by NSF RTG Grant \#1646385 and continued through his stay as a postdoctoral visitor at the Max-Planck Institut f\"{u}r Mathematik. M.L.\ acknowledges support from the Deutsche Forschungsgemeinschaft (DFG, German Research Foundation) through TRR 326 Geometry and Arithmetic of Uniformized Structures, project number 444845124.
	
	The authors would like to thank L. K\"{u}hne and Z. Gao for answers to questions about the effectivity of their results on uniformity.
	M.L.\ would also like to thank K. Gajdzica for pointing out that the constants $C_0,C_1$ in \eqref{eq:asympt-partition} can be found in the literature. D.C. would like to thank B. Mazur for helpful discussions on this paper.

	
	\section{Computing the local and global Hilbert series}
	
	The first step in our proof of Theorem~\ref{thm:mainS} is to compute the two power series
	\begin{align*}
		\HS_\loc(t) &\coloneqq  \prod_{k\geq1}(1-t^k)^{-\dim\HH^1_f(G_p;V_k)} \,, \\
		\HS_\glob(t) &\coloneqq (1-t^2)^{-s}\prod_{k\geq1}(1-t^k)^{-\dim\HH^1_f(G_\bQ;V_k)} \,,
	\end{align*}
	where~$V_k\coloneqq \gr^\W_{-k}U$ denotes the $k$th graded piece of the weight filtration on the whole~$\bQ_p$-pro-unipotent fundamental group~$U$ of~$Y$. It will turn out that the local Hilbert series~$\HS_\loc(t)$ has a closed form expression as a rational function in~$t$, depending only on the genus~$g$ and number of punctures~$n$. The global Hilbert series~$\HS_\glob(t)$ turns out to be significantly more complicated to describe: we give an expression for~$\HS_\glob(t)$ as an infinite product of roots of rational functions, whose coefficients depend not just on~$g$ and~$n$, but also on the Mordell--Weil rank~$r$ and rational N\'eron--Severi rank~$\rho$ of the Jacobian of~$Y$, the size~$s$ of the set~$S$, and the number of closed points and $\bR$-points in the boundary divisor~$D$. The computation of~$\HS_\loc(t)$ is unconditional; the computation of~$\HS_\glob(t)$ uses the Tate--Shafarevich and Bloch--Kato Conjectures.

	We begin by determining a rather more refined type of Hilbert series attached to the fundamental group~$U$. Let $K_0(\REP(G_\bQ))$ denote the Grothendieck ring of the $\otimes$-category of continuous representations of $G_\bQ$ on finite-dimensional $\bQ_p$-vector spaces. If $V$ is a graded representation of $G_\bQ$, supported in non-negative degrees and finite-dimensional in each degree, we define its Hilbert series
	\[
	[V] \coloneqq \sum_{k\geq0}[\gr_k V]t^k \in K_0(\REP(G_\bQ))\llbrack t\rrbrack \,.
	\]
	We adopt the shorthand
	\[
	\HS^\mot_U(t) \coloneqq  [\gr^\W_{\bullet} \cO(U)]\,,
	\]
	where~$\gr^\W_{\bullet}\cO(U)$ denotes the associated graded of the weight filtration on $\cO(U)$. This definition of Hilbert series generalises the usual one, in that the image of $\HS^\mot_U(t)$ under the map $\dim\colon K_0(\REP(G_\bQ))\to\bZ$ recovers the Hilbert series of~$\cO(U)$ in the usual sense.
	
	In fact, this Hilbert series can be computed explicitly.
	
	\begin{proposition}\label{prop:motivic_hilbert_series}
		\[
		\HS^\mot_U(t) = \frac1{1-[\HH^1_\et(X_{\bQbar};\bQ_p)]t-([\HH^0_\et(D_{\bQbar};\bQ_p)]-1)(-1)t^2} \,,
		\]
		where $(-1)$ denotes inverse Tate twist (multiplication by $[\bQ_p(-1)]$ in $K_0(\REP(G_\bQ))$).
		\begin{proof}
			Suppose firstly that~$D\neq\emptyset$. In this case, $U$ is free, so $\gr^\W_\bullet\cO(U)$ is $G_\bQ$-equivariantly isomorphic to the shuffle algebra on $\gr^\W_\bullet(U^\ab)^\vee=\gr^\W_\bullet\HH^1_\et(Y_{\bQbar};\bQ_p)$. We have
			\begin{align*}
				[\gr^\W_1\HH^1_\et(Y_{\bQbar};\bQ_p)] &= [\HH^1_\et(X_{\bQbar};\bQ_p)] \\
				[\gr^\W_2\HH^1_\et(Y_{\bQbar};\bQ_p)] &= ([\HH^0_\et(D_{\bQbar};\bQ_p)]-1)(-1)
			\end{align*}
			and all other graded pieces of $\HH^1_\et(Y_{\bQbar};\bQ_p)$ are zero. Hence
			\[
			\HS^\mot_U(t) = \sum_{k\geq0}\left([\HH^1_\et(X_{\bQbar};\bQ_p)t+([\HH^0_\et(D_{\bQbar};\bQ_p)]-1)(-1)t^2\right)^k
			\]
			and we are done in this case.
			\smallskip
			
			In the projective case, we let $\cU\coloneqq\gr^\W_\bullet\cO(U)^\vee$ denote the graded dual of~$\cO(U)$, which is a graded Hopf algebra. We know that $\gr^\W_\bullet\Lie(U)$ is the free Lie algebra generated by $\HH_1^\et(X_{\bQbar};\bQ_p)\coloneqq \HH^1_\et(X_{\bQbar};\bQ_p)^\vee$ in degree~$-1$, modulo the Lie ideal generated by the copy of $\bQ_p(1)$ inside ${\bigwedge}\!^2\HH_1^\et(X_{\bQbar};\bQ_p)$ induced by the dual of the cup product map. Since~$\cO(U)^\vee$ is the completed universal enveloping algebra of~$\Lie(U)$, this implies that~$\cU$ is the quotient of the tensor algebra~$T$ on~$\HH_1^\et(X_{\bQbar};\bQ_p)$ in degree~$-1$ by the two-sided ideal generated by the copy of~$\bQ_p(1)$ inside~$\HH_1^\et(X_{\bQbar};\bQ_p)^{\otimes2}$ in degree~$-2$.
			
			This in turn implies that we have an exact sequence
			\[
			T\otimes\bQ_p(1)[2]\otimes T \to J \to \cU \to \bQ_p \to 0 \,,
			\]
			where~$J\unlhd T$ is the augmentation ideal and the left-hand map is given by multiplication. Here, $V[k]$ denotes the representation~$V$ placed in degree~$-k$ for the grading. This exact sequence is in particular an exact sequence of left $T$-modules; pushing out along~$T\to\cU$ yields an exact sequence
			\[
			\cU\otimes\bQ_p(1)[2]\otimes T \to \cU\otimes\HH_1^\et(X_{\bQbar};\bQ_p)[1] \to \cU \to \bQ_p \to 0 \,,
			\]
			since~$J$ is a free left $T$-module generated by $\HH_1^\et(X_{\bQbar};\bQ_p)$ in degree~$-1$. The left-hand map above vanishes on~$\cU\otimes\bQ_p(1)[2]\otimes J$, so we obtain an exact sequence
			\begin{equation}\label{eq:projective_seq}\tag{$\ast$}
				\cU\otimes\bQ_p(1)[2] \to \cU\otimes\HH_1^\et(X_{\bQbar};\bQ_p)[1] \to \cU \to \bQ_p \to 0 \,.
			\end{equation}
			of graded $G_\bQ$-representations. We will show that the left-hand map in~\eqref{eq:projective_seq} is injective, albeit via a rather indirect argument.
			
			Let~$K$ denote the kernel of the left-hand map in~\eqref{eq:projective_seq}. Note that for an exact sequence $0 \to A_1 \to \cdots \to A_n \to 0$ of graded $G_\bQ$-representations, we have $\sum_{i=1}^n (-1)^i [A_i] = 0$. Applying this to the dual of~\eqref{eq:projective_seq} and using~$\HS_U^\mot(t)=[\cU^\vee]$ provides the identity
			\begin{equation}\label{eq:projective_grothendieck_error}\tag{$\dagger$}
				(1-[\HH^1_\et(X_{\bQbar};\bQ_p)]t+[\bQ_p(-1)]t^2)\cdot\HS^\mot_U(t) = 1 + [K^\vee]
			\end{equation}
			in $K_0(\REP(G_\bQ))\llbrack t\rrbrack$.
			
			Consider the image of the identity~\eqref{eq:projective_grothendieck_error} under the map $\dim\colon K_0(\REP(G_\bQ))\to\bZ$. The image of $\HS^\mot_U(t)$ is then just the Hilbert series $\HS_U(t)$ of~$U$ in the usual sense, and this was calculated to be $\frac1{1-2gt+t^2}$ in \cite[Lemma~2.1.8]{alex:effective}. Hence we have $\dim[K^\vee]=0\in\bZ\llbrack t\rrbrack$. But the coefficients of $[K^\vee]$ are all effective, i.e.\ are the classes of representations, so this implies that $[K^\vee]=0$. This implies that $K=0$, i.e.\ the left-hand map in~\eqref{eq:projective_seq} is injective, and~\eqref{eq:projective_grothendieck_error} gives the desired expression for~$\HS_U^\mot(t)$.
		\end{proof}
	\end{proposition}
	
	Next, we want to relate the Hilbert series $\HS^\mot_U(t)$ to the weight-graded pieces $V_k\coloneqq\gr^\W_{-k}U$ of the fundamental group. For this, we adopt the notation that if~$\nu$ is an element of a $\lambda$-ring~$R$, then
	\[
	(1-t)^{-\nu} \coloneqq \sum_{i=0}^\infty s^i(\nu)t^i \in R\llbrack t\rrbrack \,,
	\]
	where $s^i(\nu)\coloneqq (-1)^i\lambda^i(-\nu)$ denotes the $i$th symmetric power operation in~$R$. In the particular case that $R=\bZ$, then $s^i(\nu)={{\nu+i-1}\choose{i-1}}$, so this agrees with the usual meaning of $(1-t)^{-\nu}$. Note that for $V \in \REP(G_\bQ)$, we have $[\Sym(V[-k])] = (1-t^k)^{-[V]}$.
	
	\begin{proposition}\label{prop:hopf_alg_vs_lie_alg_series}
		\[
		\HS^\mot_U(t) = \prod_{k\geq1}(1-t^k)^{-[V_k^\vee]} \,.
		\]
		\begin{proof}
			Since~$U$ is isomorphic to its Lie algebra, we have that $\gr^\W_\bullet\cO(U)$ is $G_\bQ$-equivariantly isomorphic to the symmetric algebra $\Sym^\bullet(\gr^\W_\bullet\Lie(U)^\vee)$ on the $\W$-graded Lie coalgebra $\gr^\W_\bullet\Lie(U)^\vee$ of~$U$. But $\gr^\W_k\Lie(U)^\vee=V_k^\vee$, so we have
			\[
			\gr^\W_\bullet\cO(U) = \bigotimes_{k\geq1}\Sym^\bullet(V_k^\vee) \,,
			\]
			where $V_k^\vee$ is placed in degree~$k$ for the $\W$-grading. This corresponds to the given equality of power series.
		\end{proof}
	\end{proposition}
	
	\begin{remark}\label{rmk:some_V_i}
		Propositions~\ref{prop:motivic_hilbert_series} and~\ref{prop:hopf_alg_vs_lie_alg_series} allow one to recursively find universal (valid for any smooth curve over any field) formulas for $[V_k]\in K_0(\REP(G_\bQ))$ in terms of $[\HH_1^\et(X_{\bQbar};\bQ_p)]$, $[\HH_0^\et(D_{\bQbar};\bQ_p)]$, and the $\lambda$-operations by comparing coefficients. 
		
		For instance, setting $A \coloneqq [\HH_1^\et(X_{\bQbar};\bQ_p)]$ and $B \coloneqq [\HH_0^\et(D_{\bQbar};\bQ_p)(1)]-[\bQ_p(1)]$ for simplicity, we get the formulas:
		\begin{align*}
			[V_1] &= A \\
			[V_2] &= A^2 + B - s^2([V_1]) = \lambda^2(A)+B \\
			[V_3] &= A^3 + 2AB - s^3([V_1]) - [V_1][V_2] = A\lambda^2(A) + AB - \lambda^3(A) \\
			[V_4] &= A^4 + 3A^2B + B^2 - s^4([V_1]) - [V_2]s^2([V_1]) - [V_1][V_3] - s^2([V_2]) \\
			&= A^2\lambda^2(A) + A^2B -(\lambda^2(A))^2 + \lambda^2(B) \,,
		\end{align*}
		which in particular recover the well-known decompositions $[V_1] = [\HH_1^\et(X_{\bQbar};\bQ_p)]$ and $[V_2] = [\bigwedge^2 \HH_1^\et(X_{\bQbar};\bQ_p)]+[\HH_0^\et(D_{\bQbar};\bQ_p)(1)]-[\bQ_p(1)]$.\end{remark}
	
	\begin{remark}
		The relation used to derive the formulas  in Remark \ref{rmk:some_V_i} is equivalent to the relation
		\[
		[U[k]] = [\Gr^W_{-k} \Oc(U_n)^{\vee}] - [ \Gr^W_{-k}\operatorname{Sym}(\oplus_{i=1}^{k-1} U[i])] = \pr_{-k}([\Oc(U_n)^{\vee}] - [\operatorname{Sym}(\oplus_{i=1}^{k-1} U[i])])
		\]
		found in \cite[\S7.2]{CorwinMECK}. In fact, the formulas for $[V_3]$ and $[V_4]$ may be used to rederive the formulas of \cite[\S7.4-7.5]{CorwinMECK} by computing the lambda-operations in $K_0(\REP_{\Qp}^{\rm ss}(G_k,E)) \cong K_0(\REP_{\Qp}(\GL_2))$ (where $\REP_{\Qp}^{\rm ss}(G_k,E)$ is the Tannakian subcategory of $K_0(\REP(G_k))$ generated by the Tate module of a non-CM elliptic curve $E$).
	\end{remark}

	\subsection{Computing the local Hilbert series}
	
	Using Proposition~\ref{prop:motivic_hilbert_series}, we now explicitly determine the local Hilbert series~$\HS_\loc(t)$ from the introduction.
	
	\begin{lemma}\label{lem:local_series}
		\[
		\HS_\loc(t) = \frac{1-gt}{1-2gt-(n-1)t^2} \,.
		\]
	\end{lemma}
	
	In fact, this Hilbert series was already computed in \cite[Corollary~4.2.4]{alex:effective}, but we can give a different proof using Proposition~\ref{prop:motivic_hilbert_series}. For this, let $\REP_\cris^+(G_p)$ denote the subcategory of~$\REP(G_p)$ consisting of those representations~$V$ which are crystalline, have only non-negative Hodge--Tate weights, and such that all eigenvalues of the crystalline Frobenius on $\sD_\cris(V)$ are all $p$-Weil numbers of positive weight.
	
	The contravariant functors $V\mapsto\sD_\dR(V^\vee)$ and $V\mapsto\sD_\dR^+(V^\vee)$ are both exact $\otimes$-functors $\REP_\cris^+(G_p)\to\VEC_{\bQ_p}$ (the latter of which follows by Hodge strictness, cf.\ \cite[Theorem 8.2.11 and Proposition 9.2.14]{brinonconrad}), so they induce $\lambda$-ring homomorphisms $\sd_\dR,\sd_\dR^+\colon K_0(\REP_\cris^+(G_p))\to\bZ$, given by $\sd_\dR^{(+)}([V])\coloneqq\dim\sD_\dR^{(+)}(V^\vee)$. 
	\smallskip
	
	Now the restriction to~$G_p$ of all of the representations $V_k^\vee$, as well as $\HH^1_\et(X_{\bQbar};\bQ_p)$ and $\HH^0_\et(D_{\bQbar};\bQ_p)(-1)$, lie in $\REP_\cris^+(G_p)$. Hence the image of the Hilbert series $\HS^\mot_U(t)$ in $K_0(\REP(G_p))\llbrack t\rrbrack$ lies in $K_0(\REP_\cris^+(G_p))\llbrack t\rrbrack$. Applying the morphisms $\sd_\dR$ and $\sd_\dR^+$ to the identities in Propositions~\ref{prop:motivic_hilbert_series} and~\ref{prop:hopf_alg_vs_lie_alg_series} gives
	\begin{align*}
		\sd_\dR(\HS_{\mot}(t))  &= \prod_{k\geq1}(1-t^k)^{-\dim\sD_\dR(V_k)} = \frac1{1-2gt-(n-1)t^2} \,, \\
		\sd_\dR^+(\HS_{\mot}(t)) &= \prod_{k\geq1}(1-t^k)^{-\dim\sD_\dR^+(V_k)} = \frac1{1-gt} \,.
	\end{align*}
	For any~$V\in\REP_\cris^+(G_p)$, the Bloch--Kato exponential sequence gives
	\[
	\HH^1_f(G_p;V) \cong \sD_\dR(V)/\sD_\dR^+(V) \,,
	\]
	which implies that
	\[\HS_{\loc}(t) = \frac{\sd_\dR(\HS_{\mot}(t))}{\sd_\dR^+(\HS_{\mot}(t))} = \frac{1-gt}{1-2gt-(n-1)t^2}\,
	\]
	as desired. \qed
	
	\subsection{Computing the global Hilbert series}
	
	Now we turn to the rather harder problem of determining the global Hilbert series~$\HS_\glob(t)$. Our discussion parallels \cite[\S 2.1, 2.3]{CorwinMECK}. In order to carry out this computation, we use the following conjecture:
	\begin{conjecture}[{\cite[Prediction 4.1]{BellaicheNotes}}]\label{conj:BK}
		Suppose that~$V$ is a semisimple representation of~$G_\bQ$ which is de Rham at~$p$, and which is unramified and pure of weight~$w\geq 0$ at all but finitely many rational primes. Then
		\[
		\HH^1_g(G_\bQ;V)=0 \,.
		\]
		In particular, we also have $\HH^1_f(G_\bQ;V)=0$.
	\end{conjecture}
	
	This conjecture is known as a consequence of the ``Bloch--Kato conjectures'' (\cite{BellaicheNotes}), although its formulation for a general geometric Galois representation appears only in subsequent work of Fontaine--Perrin-Riou. More precisely, it follows from \cite[Conjectures 3.4.5(i)]{FPR91}, which relates $\dim \HH^1_g(G_\bQ;V)$ to $\ord_{s=0} L(V,s)$, which is in the region of absolute convergence when $w \ge 0$. When $V$ is given as a direct summand of a twist of the \'etale cohomology of a smooth proper variety (which is indeed possible in our case), \cite[Conjecture 5.3(i)]{BlochKato} would imply that $\HH^1_g(G_\bQ;V)$ comes from negative-degree algebraic $K$-theory of the variety in question, which vanishes by definition.
	
	On a more philosophical level, this conjecture follows from the idea that all geometric Galois representations should come from geometry and therefore possess a weight filtration (with semisimple weight-graded pieces, by the Grothendieck--Serre conjecture). This philosophy is described in \cite[\S 2.3.4, \S 4.1.3]{BellaicheNotes}, and a precise conjecture is \cite[``Conjecture'' 12.4]{Fontaine91}.
	
	One may use Conjecture \ref{conj:BK} to compute $\dim \HH^1_f(G_\bQ;V)$ for $w \neq -1$ when paired with the following consequence of Poitou-Tate duality:
	
	\begin{fact}[{\cite[II.2.2.2]{FPR91}}]\label{fact:PT}
		For a geometric Galois representation $V$ over a number field $k$, we have
		\begin{align*}\label{eqn:Poitou-Tate}
			\dim \HH^1_f(G_k;V) = \dim \HH^0(G_k;V) + \dim \HH^1_f(G_k;V^{\vee}(1)) - \dim \HH^0(G_k;V^{\vee}(1))\\  + \sum_{v \mid p} \dim_{\Qp} (\operatorname{D}_{\dR}{V}/\operatorname{D}^+_{\dR}{V})
			- \sum_{v \mid \infty} \dim \HH^0(G_v;V)
			.\end{align*}
	\end{fact}
	
	Finally, the only weight $-1$ case we consider is that of $V_1$, in which case the Tate--Shafarevich Conjecture computes $\dim \HH^1_f(G_\bQ;V)$. We summarise these computations in a lemma:
	
	\begin{lemma}\label{lem:global_dimensions}
		Assume the Tate--Shafarevich and Bloch--Kato Conjectures. Then
		\[
		\dim\HH^1_f(G_\bQ;V_k) = 
		\begin{cases}
			r & \text{if $k=1$,} \\
			\dim\HH^1_f(G_p;V_2)-\dim V_2^\sigma+1-\rho-\#|D| & \text{if $k=2$,} \\
			\dim\HH^1_f(G_p;V_k)-\dim V_k^\sigma & \text{if $k\geq3$.}
		\end{cases}
		\]
		where~$\sigma\in G_\bQ$ is complex conjugation (with respect to some embedding $\bQbar\hookrightarrow\bC$).
		\begin{proof}
			For~$k=1$, $V_1$ is the $\bQ_p$-linear Tate module of the Jacobian~$J$ of~$X$, so $\dim\HH^1_f(G_\bQ;V_1)$ is the $p^\infty$-Selmer rank of~$J$. Assuming the Tate--Shafarevich Conjecture, this is equal to the Mordell--Weil rank~$r$.
			
			For~$k\geq2$, we use Fact \ref{fact:PT}. We have $V_k^{G_\bQ}=0$ since $V_k$ is pure of weight~$-k$, and Conjecture \ref{conj:BK} provides the vanishing of $\HH^1_f(G_\bQ;V_k^\vee(1))$. If $k\geq3$, then also $(V_k^\vee(1))^{G_\bQ}=0$ for weight reasons, and we are done in this case.
			
			It remains to show that $\dim(V_2^\vee(1))^{G_\bQ}=\rho+\#|D|-1$. For this, note that $V_2^\vee(1)$ is a semisimple $G_\bQ$-representation \cite[Lemma~6.0.1]{alex:effective}, and we have the identity
			\[
			[V_2^\vee(1)] = [\Hom(\bigwedge\nolimits^2V_1,\bQ_p(1))] + [\bQ_p^{D(\bQbar)}] - 1
			\]
			in the Grothendieck group of semisimple $G_\bQ$-representations, as in Remark~\ref{rmk:some_V_i}. So we just need to compute the dimension of the $G_\bQ$-fixed part of both terms on the right-hand side. The latter of these is easy: the dimension of the $G_\bQ$-invariant subspace in the permutation representation $\bQ_p^{D(\bQbar)}$ is equal to the number of $G_\bQ$-orbits in $D(\bQbar)$, which is $\#|D|$.
			
			For the former, we write $\Hom(\bigwedge^2V_1,\bQ_p(1))\cong\Hom(V_1,V_1^\vee(1))^{-}$, where the superscript ${-}$ denotes those maps $\phi\colon V_1\to V_1^\vee(1)$ such that $\phi^\vee(1)=-\phi$ under the canonical identification $(V_1^\vee(1))^\vee(1)\cong V_1$. By the Tate Conjecture for abelian varieties \cite[Theorem~(b)]{Faltings83}, the canonical map
			\[
			\bQ_p\otimes_{\bZ}\Hom(J,J^\vee)^{+}\to\Hom_{G_\bQ}(V_1,V_1^\vee(1))^{-}
			\]
			is an isomorphism, where the superscript~$+$ denotes those morphisms $\phi\colon J\to J^\vee$ of abelian varieties such that $\varphi^\vee=\varphi$ under the canonical identification $J^{\vee\vee}\cong J$ (i.e.\ $\phi$ is symmetric with respect to the Rosati involution arising from the principal polarisation on~$J$). The left-hand group is isomorphic to the $\bQ$-points of the N\'eron--Severi group-scheme, so we have $\dim_{\bQ_p}\Hom_{G_\bQ}(\bigwedge^2V_1,\bQ_p(1))=\rho$ as desired.
		\end{proof}
	\end{lemma}
	
	Motivated by Lemma~\ref{lem:global_dimensions}, we introduce the Hilbert series
	\[
	\HS_\bR(t) \coloneqq \prod_{k\geq1}(1-t^k)^{-\dim V_k^\sigma} \in \bN_0\llbrack t\rrbrack \,.
	\]
	This allows us to rewrite Lemma~\ref{lem:global_dimensions} in Hilbert series form.
	
	\begin{lemma}\label{lem:global_series}
		Assume the Tate--Shafarevich and Bloch--Kato Conjectures. Then
		\[
		\HS_\glob(t) = (1-t)^{-r}\cdot(1-t^2)^{\rho+\#|D|-1-s}\cdot\frac{1-gt}{1-2gt-(n-1)t^2}\cdot\HS_\bR(t)^{-1} \,.
		\]
	\end{lemma}
	
	Unlike all the other Hilbert series we have seen, it seems that $\HS_\bR(t)$ is not in general expressible in closed form. As a replacement for an explicit expression, we prove that it satisfies a functional equation, which also happens to determine it uniquely. In the statement of this functional equation and what follows, we let~$n_1\coloneqq\# D(\bR)$ and $2n_2\coloneqq n-n_1$.
	
	\begin{lemma}\label{lem:functional_equation}
		The series $\HS_\bR(t)$ satisfies the functional equation
		\[
		\HS_\bR(t)^2 = \frac{1-2gt^2-(n-1)t^4}{(1-(n_1-1)t^2)(1-2gt-(n-1)t^2)}\cdot\HS_\bR(t^2) \,.
		\]
	\end{lemma}
	
	We will again deduce this functional equation from Propositions~\ref{prop:motivic_hilbert_series} and~\ref{prop:hopf_alg_vs_lie_alg_series}, though the argument is more complicated.
	
	Let $\REP(G_\infty)$ denote the $\otimes$-category of representations of $G_\infty=C_2$. The Grothendieck ring $K_0(\REP(G_\infty))$ is isomorphic to $\bZ[\xi]/(\xi^2-1)$, where $\xi$ is the class of the non-trivial one-dimensional representation of~$G_\infty$. There are two ring homomorphisms $K_0(\REP(G_\infty))\to\bZ$, namely $\dim\colon a+b\xi\mapsto a+b$ and $\sgn\colon a+b\xi\mapsto a-b$, whose average sends $[V]$ to $\dim V^{\sigma}$. The map~$\dim$, as the name suggests, takes the class of a $G_\infty$-representation~$V$ to $\dim(V)$, and hence $\dim$ is a homomorphism of $\lambda$-rings. The map~$\sgn$, however, is not.
	
	\begin{sublemma}\label{sublem:sgn_power}
		For all $a+b\xi\in K_0(\REP(G_\infty))$, we have the identity
		\[
		\sgn\left((1-t)^{-a-b\xi}\right) = (1-t)^{-a}(1+t)^{-b}
		\]
		of formal power series in~$t$.
		\begin{proof}
			The symmetric power operations on $K_0(\REP(G_\infty))$ are given by
			\[
			s^k(a+b\xi) = \sum_{i+j=k}s^i(a)s^j(b)\xi^j \,,
			\]
			where inside the summation $s^i(a)={{a+i-1}\choose{i-1}}$ denotes the usual symmetric power operations in the $\lambda$-ring $\bZ$. According to the definition, we have
			\[
			(1-t)^{-a-b\xi} = \sum_{k\geq0}s^k(a+b\xi)t^k = \left(\sum_{i\geq0}s^i(a)t^i\right)\left(\sum_{j\geq0}s^j(b)\xi^jt^j\right) \,.
			\]
			Applying $\sgn$ to both sides yields the claimed equality.
		\end{proof}
	\end{sublemma}
	
	\begin{proof}[Proof of Lemma~\ref{lem:functional_equation}]
		From Propositions~\ref{prop:motivic_hilbert_series} and~\ref{prop:hopf_alg_vs_lie_alg_series}, we see that the image of $\HS^\mot_U(t)$ inside $K_0(\REP(G_\infty))\llbrack t\rrbrack$ is given by
		\[
		\frac1{1-(g+g\xi)t-(n_1\xi+n_2+n_2\xi-\xi)t^2} = \prod_{k\geq1}(1-t^k)^{-[V_k]} \,.
		\]
		Applying the maps $\dim$ and $\sgn$, and using Sublemma~\ref{sublem:sgn_power} we obtain the identities
		\begin{align*}
			\frac1{1-2gt-(n-1)t^2} = \prod_{k\geq1}(1-t^k)^{-\dim V_k} 
			&= \prod_{k \ge 1} (1-t^k)^{-\dim V_k^{\sigma}}\cdot \prod_{k \ge 1} (1-t^k)^{-\dim V_k^{-\sigma}}\,, \\
			\frac1{1+(n_1-1)t^2} &= \prod_{k\geq1}(1-t^k)^{-\dim V_k^\sigma}\cdot\prod_{k\geq1}(1+t^k)^{-\dim V_k^{-\sigma}} \,.
		\end{align*}
		Taking the product of these identities gives
		\[
		\frac1{(1-(n_1-1)t^2)(1-2gt-(n-1)t^2)}=\HS_\bR(t)^2\cdot\prod_{k\geq1}(1-t^{2k})^{-\dim V_k^{-\sigma}} \,.
		\]
		Multiplying through by $\HS_\bR(t^2)$ then gives
		\[
		\frac{\HS_\bR(t^2)}{(1-(n_1-1)t^2)(1-2gt-(n-1)t^2)}=\frac{\HS_\bR(t)^2}{1-2gt^2-(n-1)t^4}
		\]
		which rearranges to the desired identity.
	\end{proof}
	
	The functional equation gives an expression for~$\HS_\bR(t)$ as an infinite product, convergent in the $t$-adic topology.
	
	\begin{lemma}\label{lem:product_expansion_of_HS_R}
		We have
		\[
		\HS_\bR(t) = \prod_{j=0}^\infty G(t^{2^j})^{1/2^{j+1}} \,,
		\]
		where
		\[
		G(t) = \frac{1-2gt^2-(n-1)t^4}{(1-(n_1-1)t^2)(1-2gt-(n-1)t^2)} \,.
		\]
		Moreover, each factor~$G(t^{2^j})^{1/2^{j+1}}$ has non-negative coefficients and constant coefficient~$1$.
		\begin{proof}
			For the first part, observe that~$\HS_\bR(t)$ is uniquely determined by the functional equation of Lemma~\ref{lem:functional_equation} and the fact it has constant coefficient~$1$. The above infinite product satisfies the same functional equation and has constant coefficient~$1$, so is equal to~$\HS_\bR(t)$.
			
			For the second part, we compute
			\begin{align*}
				\log G(t) &= \log(1-2gt^2-(n-1)t^4) - \log(1-(n_1-1)t^2) - \log(1-2gt-(n-1)t^2) \\
				&= \sum_{i=1}^\infty (\alpha^i+\beta^i)\frac{t^i}i - \sum_{i=1}^\infty(\alpha^i+\beta^i-(n_1-1)^i)\frac{t^{2i}}i \,,
			\end{align*}
			where~$\alpha$ and~$\beta$ are the roots of the polynomial $T^2-2gT-(n-1)$. We claim that~$\log G(t)$ has non-negative coefficients, i.e.\
			\[
			\alpha^i+\beta^i \geq0 \hspace{0.4cm}\text{and}\hspace{0.4cm} \alpha^{2i}+\beta^{2i}+2(n_1-1)^i \geq 2(\alpha^i+\beta^i)
			\]
			for all~$i\geq1$. For this, we consider the sequence $c_i=\frac12(\alpha^i+\beta^i)$; this has $c_0=1$, $c_1=g$ and satisfies the recurrence relation
			\begin{equation}\label{eq:recurrence_easy}\tag{$\ast$}
				c_{i+2} = 2gc_{i+1} + (n-1)c_i
			\end{equation}
			for~$i\geq0$. If~$g=0$ then $c_i=(n-1)^{i/2}$ if~$i$ is even and~$c_i=0$ if~$i$ is odd; the two claimed inequalities follow straightforwardly. If instead~$g\geq1$, then the recurrence relation~\eqref{eq:recurrence_easy} straightforwardly implies that~$(c_i)$ is a weakly increasing sequence of positive integers, validating the first inequality. It even implies that we have $c_{i+1}\geq2c_i$ for~$i\geq1$ (using hyperbolicity to rule out the possibility that~$g=1$ and~$n=0$). This implies the second inequality, except in the single case $n_1=0$ and $i=1$. In this case, the desired inequality is
			\[
			4g^2+2(n-2) \geq 4g \,,
			\]
			which is easily seen to hold by hyperbolicity and the fact that~$n$ is congruent to~$n_1$ modulo~$2$, so even.
			
			We have thus shown that~$\log G(t)$ has non-negative coefficients. It follows that
			\[
			G(t^{2^j})^{1/2^{j+1}} = \exp\left(2^{-j-1}\log G(t^{2^i})\right)
			\]
			also has non-negative coefficients for all~$j$, as desired.
		\end{proof}
	\end{lemma}
	
	As a consequence, we can give explicit approximations to $\HS_\glob(t)$ by truncating the product expansion of~$\HS_\bR(t)$ at some finite stage. The approximation we will use in the rest of paper is below.
	
	\begin{corollary}\label{cor:global_bound}
		Assuming the Tate--Shafarevich and Bloch--Kato Conjectures, we have the coefficientwise upper bound
		\[
		\HS_\glob(t) \leq (1-t)^{-r}\cdot(1-t^2)^{\rho+\#|D|-1-s}\cdot\frac{1-gt}{\sqrt{1-2gt-(n-1)t^2}}\cdot\sqrt{\frac{1-(n_1-1)t^2}{1-2gt^2-(n-1)t^4}} \,.
		\]
		\begin{proof}
			This is just the inequality
			\[
			\HS_\glob(t) \leq \HS_\glob(t)\cdot\prod_{j=1}^\infty G(t^{2^j})^{1/2^{j+1}}
			\]
			where~$G$ is as in Lemma~\ref{lem:product_expansion_of_HS_R} and we use the expression for $\HS_\glob(t)$ from Lemma~\ref{lem:global_series}.
		\end{proof}
	\end{corollary}
	
	\section{Proof of the main theorem}
	
	Now that we have determined the Hilbert series $\HS_\loc(t)$ and $\HS_\glob(t)$, we are in a position to prove Theorems~\ref{thm:main} and~\ref{thm:mainS}. In order to apply the bounds in \cite{alex:effective}, we need to determine a non-negative integer~$m$ for which the inequality
	\begin{equation}\label{eq:coeff_ineq_again}\tag{$\dagger$}
		\sum_{i=0}^mc_i^\glob < \sum_{i=0}^m c_i^\loc
	\end{equation}
	holds, with~$c_i^\loc$ and~$c_i^\glob$ the coefficients of~$\HS_\loc(t)$ and~$\HS_\glob(t)$. In other words, we would like the $m$th coefficient of~$\frac1{1-t}\HS_\glob(t)$ to be strictly less than the $m$th coefficient of~$\frac1{1-t}\HS_\loc(t)$.
	
	The observation underlying the proof is the following. Lemma~\ref{lem:local_series} tells us that~$\frac1{(1-t)^2}\HS_\loc(t)^2$ is a rational function in~$t$, and that its closest pole to~$0$ is a double pole at~$1/\alpha$ where~$\alpha$ is the largest root\footnote{In the case~$g=0$, this polynomial has two roots of equal magnitude. For simplicity, we ignore this case in the exposition here.} of the polynomial $T^2-2gT-(n-1)$. So the coefficients of~$\frac1{(1-t)^2}\HS_\loc(t)^2$ grow like $i\alpha^i$. On the other hand, Corollary~\ref{cor:global_bound} tells us that $\frac1{(1-t)^2}\HS_\glob(t)^2$ is bounded above coefficientwise by a rational function in~$t$ whose closest pole to~$0$ is a simple pole at~$1/\alpha$. So the coefficients of $\frac1{(1-t)^2}\HS_\glob(t)^2$ grow at most like~$\alpha^i$. Hence there is some $m\gg0$ such that the $m$th coefficient of $\frac1{(1-t)^2}\HS_\glob(t)^2$ is strictly less than the $m$th coefficient of $\frac1{(1-t)^2}\HS_\loc(t)^2$, and this immediately implies the same without the squares (for a possibly smaller value of~$m$).
	
	Making this quantitative involves some careful manipulation of inequalities, which we carry out in this section. We begin with some elementary inequalities (Lemmas \ref{lem:sum_coeffs} and \ref{lem:functional_bound}), which we use to give an upper bound on the coefficients of $\frac1{1-t}\HS_\glob(t)$ (Lemma \ref{lem:global_bound}). Then we use this bound to control the value of~$m$ for which the $m$th coefficient of~$\frac1{1-t}\HS_\loc(t)$ exceeds the $m$th coefficient of $\frac1{1-t}\HS_\glob$, and conclude using \cite[Theorem~B]{alex:effective}. We make no particular effort to make the inequalities in this section as tight as possible, though we believe the final bound on~$m$ is asymptotically correct.
	
	In the course of proving these inequalities, we use without comment the fact that the power series~$\frac{1-gt}{1-2gt-(n-1)t^2}$ has non-negative coefficients. This follows from Lemma~\ref{lem:local_series} or by direct computation. Hence~$\frac1{1-2gt-(n-1)t^2}=\frac1{1-gt}\cdot\frac{1-gt}{1-2gt-(n-1)t^2}$ also has non-negative coefficients. We use the symbol~$\leq$ to denote coefficientwise inequality of power series.

	\begin{lemma}\label{lem:sum_coeffs}
		Let~$\delta=1$ if~$g\geq1$ and~$\delta=2$ if~$g=0$. Then we have the coefficientwise inequality
		\[
		\frac1{1-t^\delta}\cdot\frac{1-gt}{1-2gt-(n-1)t^2} \leq 2\cdot\frac{1-gt}{1-2gt-(n-1)t^2} \,.
		\]
		\begin{proof}
			We want to show that the power series
			\[
			F(t) \coloneqq  \frac{(1-2t^\delta)(1-gt)}{(1-t^\delta)(1-2gt-(n-1)t^2)}
			\]
			has non-negative coefficients. When~$g=0$, we have
			\[
			F(t) = \frac1{1-t^2}\cdot\left(1+\frac{(n-3)t^2}{1-(n-1)t^2}\right) \,.
			\]
			Each of the two factors has non-negative coefficients since~$n\geq3$, and hence so too does~$F$. When~$g=1$, we have
			\[
			F(t) = \frac{1-2t}{1-2t-(n-1)t^2} = 1+\frac{(n-1)t^2}{1-2t-(n-1)t^2}
			\]
			which clearly has non-negative coefficients since~$n\geq1$. When~$g\geq2$, we have
			\[
			F(t) = \frac1{1-t^2}\cdot\left(1+\frac{(g-2)t}{1-gt}\right)\cdot\left(1+\frac{(g^2+n-1)t^2}{1-2gt-(n-1)t^2}\right)
			\]
			which clearly has non-negative coefficients, since it is a product of three power series with non-negative coefficients.
		\end{proof}
		
		A more sophisticated version of the same kind of analysis proves the following bound.
		
	\end{lemma}
	
	\begin{lemma}\label{lem:functional_bound}
		We have the coefficientwise inequality
		\[
		\frac{(1-gt)^2\cdot(1-(n_1-1)t^2)}{(1-2gt-(n-1)t^2)\cdot(1-2gt^2-(n-1)t^4)} \leq 2\cdot\frac{1-gt}{1-2gt-(n-1)t^2} \,.
		\]
		\begin{proof}
			We wish to show that the power series
			\[
			F(t)\coloneqq \left(2-\frac{(1-gt)(1-(n_1-1)t^2)}{1-2gt^2-(n-1)t^4}\right)\cdot\frac{1-gt}{1-2gt-(n-1)t^2}
			\]
			has non-negative coefficients. For this, we write
			\begin{align*}
				F(t) &= \frac{(1-gt)\cdot(1+3gt+2t^2)}{1-2gt^2-(n-1)t^4} \\
				&+\frac{1-gt}{1-2gt-(n-1)t^2}\cdot\frac{(6g^2-4g+n+n_1-4)t^2+g(3n-n_1+2)t^3}{1-2gt^2-(n-1)t^4}
			\end{align*}
			By inspection, the coefficients $6g^2-4g+n+n_1-4$ and $g(3n-n_1+2)$ are both non-negative (using that~$n_1\leq n$ has the same parity as~$n$). Since~$\frac{1-gt}{1-2gt-(n-1)t^2}\geq1$, we thus have the coefficientwise inequality
			\[
			F(t) \geq \frac{1+2gt+(3g^2-4g+n+n_1-2)t^2+g(3n-n_1)t^3}{1-2gt^2-(n-1)t^4} \,.
			\]
			If~$g\neq1$, then the numerator of the right-hand side has non-negative coefficients and so~$F(t)\geq0$ as desired. If~$g=1$, then the numerator is~$\geq1-t^2$, and hence
			\[
			F(t) \geq \frac{1-t^2}{1-2t^2-(n-1)t^4} \geq 0
			\]
			as desired.
		\end{proof}
	\end{lemma}
	
	Using Lemmas~\ref{lem:sum_coeffs} and~\ref{lem:functional_bound}, we obtain the following bound on the power series~$\HS_\glob(t)$.
	
	\begin{lemma}\label{lem:global_bound}
		We have the coefficientwise inequality
		\[
		\frac1{(1-t^\delta)^2}\HS_\glob(t)^2 \leq 2^{2r+2\bar s+3}\frac{1-gt}{1-2gt-(n-1)t^2} \,,
		\]
		where $\bar s\coloneqq \max\{s+1-\rho,0\}$ and~$\delta\in\{1,2\}$ is as in Lemma~\ref{lem:sum_coeffs}.
		\begin{proof}
			Squaring the bound from Corollary~\ref{cor:global_bound} gives the componentwise inequality
			\[
			\frac1{(1-t^\delta)^2}\HS_\glob(t)^2 \leq \frac1{(1-t^\delta)^{2+2r+2\bar s}}\cdot \frac{(1-gt)^2\cdot(1-(n_1-1)t^2)}{(1-2gt-(n-1)t^2)\cdot(1-2gt^2-(n-1)t^4)} \,.
			\]
			So the claimed bound follows from Lemmas~\ref{lem:sum_coeffs} and~\ref{lem:functional_bound}.
		\end{proof}
	\end{lemma}
	
	Using this bound, we can complete the proof of Theorem~\ref{thm:main}. Recall that Theorem~\ref{thm:main} is a consequence of \cite[Theorem~1.1.1]{alex:effective} once we have verified the following.
	
	\begin{lemma}\label{lem:m_bound}
		There is some $m\leq M\coloneqq 4^{r+\bar s+2}$ such that the coefficient of $t^m$ in $\frac1{1-t}\HS_\glob(t)$ is strictly less than the corresponding coefficient of $\frac1{1-t}\HS_\loc(t)$.
		\begin{proof}
			We prove that there is some $m\leq M$ such that the coefficient of $t^m$ in $\frac1{1-t^\delta}\HS_\glob(t)$ is strictly less than the corresponding coefficient of~$\HS_\loc(t)$, where~$\delta=1$ or~$\delta=2$ according as $g\geq1$ or~$g=0$. Since $\HS_\glob(t)$ and $\HS_\loc(t)$ are both power series in~$t^2$ when $g=0$, the stated result is easily deduced from this.
			
			We argue by contradiction. Let us write $\leq_M$ for the partial ordering on power series where $\sum_ia_it^i\leq_M\sum_ib_it^i$ just when $a_i\leq b_i$ for $i\leq M$. We thus assume for contradiction that $\HS_\loc(t)\leq_M\frac1{1-t^\delta}\HS_\glob(t)$. Since $\HS_\glob(t)$ and $\HS_\loc(t)$ are power series with non-negative coefficients, we may square both sides and use Lemma~\ref{lem:global_bound} to deduce that
			\begin{equation}\label{eq:to_contradict}\tag{$\ast$}
				\left(\frac{1-gt}{1-2gt-(n-1)t^2}\right)^2 \leq_M 2^{2r+2\bar s+3}\frac{1-gt}{1-2gt-(n-1)t^2} \,.
			\end{equation}
			Now let~$\alpha,\beta$ denote the roots of the polynomial $T^2-2gT-(n-1)=0$, so that we have the partial fractions expansion
			\[
			\frac{1-gt}{1-2gt-(n-1)t^2} = \frac12\left(\frac1{1-\alpha t}+\frac1{1-\beta t}\right) \,.
			\]
			Hence the coefficients\footnote{These were also denoted~$c_i^\loc$ in the introduction.}~$c_i$ of $\frac{1-gt}{1-2gt-(n-1)t^2}$ are given by $c_i=\frac12(\alpha^i+\beta^i)$. Squaring the above equality shows that
			\begin{align*}
				\left(\frac{1-gt}{1-2gt-(n-1)t^2}\right)^2 &= \frac14\left(\frac1{(1-\alpha t)^2}+\frac2{1-2gt-(n-1)t^2}+\frac1{(1-\beta t)^2}\right) \\
				&\geq \frac14\left(\frac1{(1-\alpha t)^2}+\frac{2-2gt}{1-2gt-(n-1)t^2}+\frac1{(1-\beta t)^2}\right) \,,
			\end{align*}
			so the coefficients $c^{(2)}_i$ of $\left(\frac{1-gt}{1-2gt-(n-1)t^2}\right)^2$ satisfy
			\[
			c^{(2)}_i\geq\frac14\left((i+1)\alpha^i+2c_i+(i+1)\beta^i\right)=\frac{i+2}2c_i \,.
			\]
			In particular, the~$M$th coefficient on the left-hand side of~\eqref{eq:to_contradict} is at least $\frac{M+2}2c_M$ while the corresponding coefficient on the right-hand side is $2^{2r+2\bar s+3}c_M$. Since~$M$ is even, we have that $c_M>0$, and $\frac{M+2}2>2^{2r+2\bar s+3}$ by choice of~$M$. This contradicts the inequality~\eqref{eq:to_contradict}, and completes the proof of the lemma.
		\end{proof}
	\end{lemma}
	
	Now to prove Theorem~\ref{thm:mainS}, and hence Theorem~\ref{thm:main} as a special case. Plugging the bound from Lemma~\ref{lem:m_bound} into \cite[Theorem~6.2.1]{alex:effective}, we obtain the upper bound
	\[
	\#\Y(\bZ_p)_{S,\infty} \leq \kappa_p\cdot\#\Y(\bF_p)\cdot\prod_{\ell\notin S}n_\ell\cdot\prod_{\ell\in S}(n_\ell+n)\cdot(4g+2n-2)^M\cdot\prod_{i=0}^{M-1}(c_i+1)
	\]
	where~$M=4^{r+\bar s+2}$ and $(c_i)$ are the coefficients of the power series~$\HS_\loc(t)=\frac{1-gt}{1-2gt-(n-1)t^2}$, as above. These coefficients satisfy $c_0=1$, $c_1=g$ and
	\[
	c_i = 2gc_{i-1} + (n-1)c_{i-2}
	\]
	for~$i\geq2$. The sequence~$(c_i)$ consists of non-negative integers, so an easy induction shows that $c_i\leq(2g+n)^i$ for~$i\geq0$. Hence
	\[
	\prod_{i=0}^{M-1}(c_i+1) \leq (2g+n)^{(M^2-M)/2}
	\]
	and we have proved Theorem~\ref{thm:mainS}.\qed
	
	\begin{remark}
		Although one needs only an upper bound on~$m$ to deduce Theorem~\ref{thm:main}, it is natural to wonder also about lower bounds, to see how far the argument could be optimised. We suspect, based on the computations in \cite[Remark~7.0.4]{alex:effective}, that the minimal~$m$ should grow like $4^{r+s}$ as a function of~$r$ and~$s$, for fixed~$g$ and~$n$. In other words the upper bound on~$m$ from Lemma~\ref{lem:m_bound} should be of the correct order of magnitude.
		
		Although we will not prove this in general, we sketch an argument below which gives an exponential lower bound on~$m$ in the special case of a once-punctured elliptic curve.
	\end{remark}
	
	\begin{lemma}
		Suppose that~$g=n=1$. Then we have
		\[
		\sum_{i=0}^mc_i^\glob \geq \sum_{i=0}^m c_i^\loc
		\]
		whenever
		\[
		20 \leq m \leq \frac1{e^{1/21}\pi}(5/4)^{2r+2s-2} \,.
		\]
		If~$r\geq1$ and~$r+s\gg0$, then the same inequality also holds for~$0\leq m\leq 20$.
		\begin{proof}[Proof (sketch)]
			The case~$r=s=0$ is vacuous, so we may assume that~$r+s\geq1$. The product expansion of Lemma~\ref{lem:product_expansion_of_HS_R} implies that we have
			\begin{align*}
				\frac1{1-t}\HS_\glob(t) &= (1-t)^{-r}\cdot(1-t^2)^{1-s}\cdot\prod_{j=0}^\infty(1-2t^{2^j})^{-1/2^{j+1}} \\
				&\geq \frac{(1+t^2)^{r+s-1}}{\sqrt{1-2t}} = \left(\sum_{i\geq0}{{r+s-1}\choose{i}}t^{2i}\right)\cdot\left(\sum_{j\geq0}\frac1{2^j}{{2j}\choose{j}}t^j\right)
			\end{align*}
			The bounds on Stirling's approximation from \cite{robbins:stirling} imply that one has the inequality
			\[
			{{2j}\choose{j}} \geq \frac{4^j}{e^{1/6j(6j+1)}\sqrt{\pi j}} \geq \frac{4^j}{e^{1/42}\sqrt{\pi j}}
			\]
			for all~$j\geq1$, and hence for~$m\geq1$ the $m$th coefficient of $\frac1{1-t}\HS_\glob(t)$ is
			\[
			\sum_{i=0}^mc_i^\glob \geq \sum_{i=0}^{\lfloor m/2\rfloor}\frac1{2^{m-2i}}{{r+s-1}\choose{i}}{{2m-4i}\choose{m-2i}} \geq \frac{(1+1/4)^{\bar m}\cdot 2^m}{e^{1/42}\sqrt{\pi m}}
			\]
			where~$\bar m = \min(\lfloor m/2\rfloor,r+s-1)$. It is easy but laborious to check that this estimate implies that
			\[
			\sum_{i=0}^mc_i^\glob \geq \sum_{i=0}^m c_i^\loc = 2^m
			\]
			for the claimed values of~$m$.
		\end{proof}
	\end{lemma}
	
	\appendix
	
	\section{$S$-integral points on punctured CM elliptic curves}
	
	In \cite{kimCM} Kim proved finiteness of the set of $S$-integral points of once-punctured elliptic curves $\mathcal{Y}$ with complex multiplication (CM) by computing dimensions of Selmer groups of the graded pieces of a particular quotient $U'$ of the $\bQ_p$-pro-unipotent fundamental group $U$.
	The effective Chabauty--Kim Theorem \cite[Theorem 6.2.1]{alex:effective}, which is the main input for proving the main result Theorem \ref{thm:mainS} above, is formulated for any quotient of $U$ and gives an upper bound for the size of $\mathcal{Y}(\mathbb{Z}_p)_{S,U}$.
	The aim of this section is to apply our method with $U'$ in place of $U$ and compare the resulting bounds, for punctured CM elliptic curves.
	We will see that, for $\# S$ large, the bounds for $U'$ are better than the bounds for $U$, even though \emph{a priori} we have the containment $\mathcal{Y}(\mathbb{Z}_p)_{S,U} \subseteq \mathcal{Y}(\mathbb{Z}_p)_{S,U'}$.
	Along the way, we will assume only a conjecture about non-vanishing of $p$-adic $L$-functions, but could also assume Conjecture \ref{conj:BK} instead.
	We also prove an unconditional result, see Remark \ref{rem:boundCMECuncond}.
	
	So let $E$ be an elliptic curve over $\bQ$ with CM by an imaginary quadratic number field $K$, and let $p$ be a prime of good reduction that splits in $K$, that is $p=\pi\bar{\pi}$ with $\pi, \bar{\pi}\in K$.
	Let $\mathcal{E}$ denote the minimal Weierstra\ss ~model of $E$ over $\bZ$.
	Let $Y$ be the punctured elliptic curve $E\setminus\{0\}$ and $\Y$ be the model of $Y$ that is the complement in $\mathcal{E}$ of the closure of $0$.
	To apply the effectivity results, we need a regular model of $E\setminus{0}$.
	So let $\mathcal{E}_{\min}$ be the minimal regular model of $E$ over $\bZ$ and $\mathcal{Y}_{\min}$ the complement in $\mathcal{E}_{\min}$ of the closure of $0$.
	Let $S$ denote a finite set of primes of $\bQ$ with $p\not\in S$.
	Note that $\mathcal{Y}_{\min}(\bZ_S) = \mathcal{Y}(\bZ_S)$ by the construction of the minimal regular model from a Weierstra\ss ~model, compare Tate's algorithm \cite[\S IV.9]{SilvermanArithEll2}.
	Thus bounding $\Y_{\min}(\bZ_p)_{S,U'}$ automatically bounds the number of $S$-integral points on the affine Weierstra\ss ~model $\Y$.
	
	Let $U$ denote the $\bQ_p$-pro-unipotent fundamental group of $Y$.
	Its Lie algebra $L$ is the pro-nilpotent completion of the free Lie algebra in two generators $e,f$, where we choose $e$ (resp.\ $f$) as a lift of a basis of $V_\pi(E)$ (resp.\ $V_{\bar{\pi}}(E)$) in $L^{\ab} = V_p(E) = V_\pi(E)\oplus V_{\bar{\pi}}(E)$.
	Let $L'$ be the quotient of $L$ by the pro-nilpotent completion of the span of all Lie monomials in $e,f$ in which both $e$ and $f$ appear at least twice.
	Finally, let $U'$ be the quotient of $U$ whose Lie algebra is $L'$ ($U'$ is denoted $W$ in \cite{kimCM}).
	
	The quotient $U'$ of $U$ inherits a weight filtration which we also denote by $\W$.
	The corresponding graded pieces are denoted by $V'_k \coloneqq  \gr_{-k}^{\W} U'$ and the corresponding local and global Hilbert series by $\HS'_\loc(t)$ and $\HS'_\glob(t)$, respectively.
	Theorem \cite[Theorem 6.2.1]{alex:effective} provides an explicit bound on the number of $S$-integral points of $\Y$ once we find an explicit $m$ such that the coefficient of $t^m$ in $\frac{1}{1-t}\HS'_\glob(t)$ is strictly smaller than the corresponding coefficient of $\frac{1}{1-t}\HS'_\loc(t)$.
	
	The weight filtration on $U'$ agrees with the filtration given by the descending central series.
	The following lemma summarises \cite[\S 1]{kimCM}.
	
	\begin{lemma}\label{lem:quotgalmodule}
		Let $V_p(E)$ be the rational $p$-adic Tate module of $E$, and let $\chi\colon G_K \to \bQ_p^\times$ (resp.\ $\bar{\chi}$) encode the Galois action on the $\pi$-adic (resp.\ $\bar{\pi}$-adic) Tate module $V_\pi(E) = (\varprojlim_n E(\overline{\bQ})[\pi^n])\otimes_{\bZ_p}\bQ_p$ (resp.\ $V_{\bar{\pi}}(E)$).
		Then
		\[
		V'_k \cong \begin{cases}
			V_p(E) \cong \bQ_p(\chi)\oplus\bQ(\bar{\chi}), & k=1,\\
			\bQ_p(1), & k=2,\\
			\bQ_p(\chi^{n-2}(1))\oplus\bQ_p(\bar{\chi}^{n-2}(1)), & k\geq 3,
		\end{cases}
		\]
		where the isomorphisms for $k\geq 3$ and the second isomorphism for $k=1$ are isomorphisms of $G_K$-modules.
	\end{lemma}
	
	Using Lemma \ref{lem:quotgalmodule}, we extract the dimensions of the local and global Selmer groups.
	\begin{lemma}\label{lem:CMEClocSelmerdim}
		For the local Selmer groups, we have
		\[
		\dim \HH^1_f(G_p;V'_k) = \begin{cases}
			1, & k=1,\\
			1, & k=2,\\
			2, & k\geq 3.
		\end{cases}
		\]
	\end{lemma}
	
	\begin{proof}
		As $\HH^0(G_p;V'_k)=0$ for $k\geq 1$, it follows from \cite[Corollary 3.8.4]{BlochKato} that $\dim \HH^1_f(G_p;V'_k)$ is equal to the number of negative Hodge--Tate weights of $V'_k$.
		For $k=1$, the Tate module $V_p(E)$ has Hodge--Tate weights $0$ and $-1$, each with multiplicity $1$, from which the result follows.
		For $k=2$, the result is clear.
		For $k\geq 3$, note that $\bQ_p(\chi^{n-2}(1))\oplus\bQ_p(\bar{\chi}^{n-2}(1))$ is a subrepresentation of $V_p(E)^{\otimes k-2}(1)$, and all Hodge--Tate weights of the latter are negative, thus the result follows.
	\end{proof}
	
	From now on, we will assume the Tate--Shafarevich Conjecture and the following non-vanishing result for $p$-adic $L$-functions as in the hypothesis of \cite[Theorem 0.2]{kimCM}.
	\begin{conjecture}\label{conj:non-vani-padicLfct}
		The $p$-adic $L$-functions attached to the $\pi$- (resp.\ $\bar{\pi}$-)power-torsion of $E/K$ do not vanish at negative powers of $\chi$ (resp.\ $\bar{\chi}$). As mentioned in \cite[p.2]{kimCM}, this is a conjecture in folklore.
	\end{conjecture}
	
	\begin{lemma}\label{lem:CMECglobSelmerdim}
		Assume the Tate--Shafarevich Conjecture and Conjecture \ref{conj:non-vani-padicLfct}.
		For the global Selmer groups, we then have
		\[
		\dim \HH^1_f(G_\bQ;V'_k) = \begin{cases}
			r, & k=1,\\
			0, & k=2,\\
			1, & k\geq 3,
		\end{cases}
		\]
		where $r$ is the rank of $E(\bQ)$.
	\end{lemma}
	
	\begin{proof}
		The case $k=1$ holds by the Tate--Shafarevich Conjecture.
		For $k=2$, use that $\HH^1_f(G_\bQ;\bQ_p(1)) \cong \bZ^\times\otimes_\bZ\bQ_p = 0$, which follows by \cite[Proposition 5.4]{BlochKato} for $A=\Gm$.
		For $k\geq 3$, note that $\HH^1_f(G_\bQ;V'_k)$ embeds into $\HH^1_{f,T\setminus\{p\}}(G_{\bQ,T};V'_k)$, where $T$ is any finite set of primes containing $S\cup\{p\}$ and all primes where $E$ has bad reduction, and $\HH^1_{f,T\setminus\{p\}}(G_{\bQ,T};V'_k)$ denotes the cohomology classes which are crystalline at~$p$ and unramified outside~$T$. Conjecture~\ref{conj:non-vani-padicLfct} and \cite[p.\ 6, after Claim 3.2]{kimCM} imply that the latter has dimension one, so $\dim \HH^1_f(G_\bQ;V'_k) \leq 1$.
		This is all we need in what follows, but one can actually show that $\dim \HH^1_f(G_\bQ;V'_k) = 1$ using Fact \ref{fact:PT}.
	\end{proof}
	
	\begin{remark}\label{rem:globSeldimonelargek}
		\begin{enumerate}
			\item Without assuming the Tate--Shafarevich Conjecture nor Conjecture \ref{conj:non-vani-padicLfct}, one gets $\dim \HH^1_f(G_\bQ;V'_1) = r_p$, where $r_p$ is the $p^\infty$-Selmer rank of $E$, and $\HH^1_f(G_\bQ;V'_2) = 0$. Moreover, \cite[\S 3]{kimCM} still yields the bound $\dim \HH^1_f(G_\bQ;V'_k) \leq 1$ (and hence $=1$) for $k$ large enough.
			\item The Bloch--Kato Conjecture \ref{conj:BK} also implies the dimension formulae of Lemma \ref{lem:CMECglobSelmerdim} for $k\geq 3$ by the argument of Lemma \ref{lem:global_dimensions}.
		\end{enumerate}
	\end{remark}
	
	\begin{corollary}
		Let $s=\# S$.
		Assuming the Tate--Shafarevich Conjecture and Conjecture \ref{conj:non-vani-padicLfct}, we have
		\begin{align*}
			\HS'_\loc(t) &= (1-t)^{-1} (1-t^2)^{-1} \prod_{k\geq 3} (1-t^k)^{-2},\\
			\HS'_\glob(t) &= (1-t)^{-r} (1-t^2)^{-s} \prod_{k\geq 3} (1-t^k)^{-1}.
		\end{align*}
	\end{corollary}
	
	\begin{proof}
		By definition of $\HS'_\loc$ and $\HS'_\glob$ using Lemmas \ref{lem:CMEClocSelmerdim} and \ref{lem:CMECglobSelmerdim}.
	\end{proof}
	
	To apply the effective Chabauty--Kim Theorem \cite[Theorem 6.2.1]{alex:effective}, we need to find an upper bound for the smallest integer $m$ such that the coefficient $A_m$ of $\frac{1}{1-t}\HS'_\glob(t)=\sum_{n\geq 0} A_n t^n$ is strictly smaller than the corresponding coefficient $B_m$ of $\frac{1}{1-t}\HS'_\loc(t) = \sum_{n\geq 0} B_n t^n$.
	In the following paragraphs, we give a bound on $m$, following a strategy inspired by \cite[\S 9.6]{BrownIntegral} based on the asymptotics of partition numbers.
	First of all, it is enough to compare the coefficients of the power series
	\begin{align*}
		\sum_{n\geq 0} b_n t^n &\coloneqq  \frac{1}{1-t^2} \frac{1}{1-t}\HS'_\loc(t) = f(t)^2,\\ 
		\sum_{n\geq 0} a_n t^n &\coloneqq  \frac{1}{1-t^2}\frac{1}{1-t}\HS'_\glob(t) = (1-t)^{-r} (1-t^2)^{-s} f(t),
	\end{align*}
	where
	\[
	f(t)=\prod_{k\geq 1} (1-t^k)^{-1} = \sum_{n\geq 0} p(n) t^n,
	\]
	with $p(n)$ being the partition function.
	Indeed, if $a_m < b_m$, then the formula $a_m=A_m + A_{m-2}+\dots+A_{((-1)^{m+1}+1)/2}$ and its analogue for $b_m$ show that there exists some $m'\leq m$ with $A_{m'} < B_{m'}$ as desired.
	To compare $a_n$ and $b_n$, let us first observe that on the global side we certainly have $a_n\leq \widetilde{a}_n$, where
	\[
	\sum_{n\geq 0} \widetilde{a}_n t^n = (1-t)^{-r'} f(t), \quad r'=r+s.
	\]
	We assume from now on that $r'\geq 2$.
	
	\begin{theorem}\label{thm:boundCMECm}
		There exists a constant $C_4>0$ such that for all $r'\geq 2$ there exists an $m\leq C_4 \cdot (r')^2 \cdot \log(r')^2$ with $\widetilde{a}_m < b_m$.
	\end{theorem}
	
	\begin{remark}
		Computations with SageMath suggest that the bound of Theorem \ref{thm:boundCMECm} is close to optimal.
	\end{remark}
	
	The proof needs the following three lemmas.
	We will make heavy use of the asymptotic behaviour of the partition numbers \cite[(1.41)]{hardyramanujan}, which implies the existence of two real constants $C_0,C_1>0$ such that
	\begin{equation}\label{eq:asympt-partition}
		C_0 \frac{1}{n} e^{\pi\sqrt{2n/3}} \leq p(n) \leq C_1 \frac{1}{n} e^{\pi\sqrt{2n/3}} \quad \text{for all } n\in \NN.
	\end{equation}
	
	\begin{lemma}\label{lem:global-asympt}
		The global coefficients satisfy
		\[
		\widetilde{a}_n \leq 2^{r'} C_1 n^{r'-1} e^{\pi\sqrt{2n/3}} \quad \text{for all } n\in\NN.
		\]
	\end{lemma}
	
	\begin{proof}
		By induction on $r'$, the coefficient $\widetilde{a}_n$ is a sum of $(n+1)^{r'}$ many partition numbers $p(k)$ with $k\leq n$.
		The partition numbers form an increasing sequence, so together with \eqref{eq:asympt-partition} we get the bound
		\[
		\widetilde{a}_n \leq (n+1)^{r'} p(n) \leq C_1 \frac{(n+1)^{r'}}{n} e^{\pi\sqrt{2n/3}}.
		\]
		As $\frac{(n+1)^{r'}}{n^{r'}} \leq 2^{r'}$, the result follows.
	\end{proof}
	
	\begin{lemma}\label{lem:locasymp}
		There is a constant $C_2>0$ such that the local coefficients $b_n$ satisfy
		\[
		b_n \geq C_2\frac{1}{n^2} e^{\pi\sqrt{4n/3}} \quad \text{for all } n\in\NN.
		\]
	\end{lemma}
	
	\begin{proof}
		By definition, we have
		\[
		b_n = \sum_{k=0}^n p(n) p(n-k),
		\]
		which is greater or equal to the middle term(s) of the sum, i.\,e.\ if $n$ is even we have $b_n \geq p(\frac{n}{2})^2$ and if $n$ is odd we have $b_n \geq 2p(\frac{n-1}{2})p(\frac{n+1}{2})$.
		Plugging in the lower bound in \eqref{eq:asympt-partition} immediately yields the desired bound if $n$ is even, and if $n$ is odd we use $2\sqrt{n}-(\sqrt{n-1}+\sqrt{n+1}) \to 0$ as $n\to \infty$ to adjust the exponents as well.
	\end{proof}
	
	It still remains to find an $m$ such that $\widetilde{a}_m<b_m$.
	Using Lemmas \ref{lem:global-asympt} and \ref{lem:locasymp}, we should compare the functions $\widetilde{a}(x)\coloneqq 2^{r'} C_1 x^{r'-1} e^{\pi\sqrt{2x/3}}$ and $b(x) \coloneqq  C_2 x^{-2} e^{\pi\sqrt{4x/3}}$ for $x>0$.
	Let us look at their logarithms $\alpha(x) \coloneqq  \log \widetilde{a}(x)$ and $\beta(x) \coloneqq  \log b(x)$.
	
	\begin{lemma}\label{lem:localvsglobalvialog}
		There is a constant $C_3>0$ such that for $r'\geq 2$ we have
		\[
		\alpha(x) < \beta(x) \quad \text{for all } x\geq C_3 \cdot (r')^2 \cdot \log(r')^2.
		\]
	\end{lemma}
	
	\begin{proof}
		First of all, put $u = \sqrt{x}$ and $\tau \coloneqq  \left(\sqrt{2}-1\right)\pi\sqrt{\frac{2}{3}} >1$ and look at 
		\[
		\gamma(u) \coloneqq  \beta(u^2) - \alpha(u^2) = \tau u - 2(r'+1)\log(u) + (\log C_2 - \log C_1 - r'\log 2)
		\]
		It is enough to show that there is a constant $c_3>0$ such that $\gamma(u) > 0$ for all $u\geq c_3 r' \log(r')$.
		First observe that for $u\geq 4r'$ we have $\frac{d\gamma}{du}(u) = \tau - 2\frac{r'+1}{u} \geq \tau - 2 \frac{2r'}{4r'} > 0.06$, thus $\gamma(u)$ is increasing for $u\geq 4r'$.
		Secondly, note that
		\begin{equation}\label{seq:conv}
			\frac{2(r'+1)(\log r' + \log \log r') - (\log C_2 - \log C_1 - r'\log 2)}{\tau r' \log r' - 2(r'+1)} \longrightarrow \frac{2}{\tau}, \quad r' \longrightarrow \infty.
		\end{equation}
		In particular, there exists a constant $c_1>0$ such that 
		\begin{equation}\label{eq:numdenombound}
			2(r'+1)(\log r' + \log \log r') - (\log C_2 - \log C_1 - r'\log 2) < c_1 (\tau r' \log r' - 2(r'+1))
		\end{equation}
		for all $r'\geq 9$ (so that the denominator in \eqref{seq:conv} is positive).
		With $c_2 = \max\{4,c_1\}$ we get for $r'\geq 9$ and $u\geq c_2 r' \log r'$ that, using that $\gamma$ is increasing followed by $c_2 \geq \log(c_2)$ and finally \eqref{eq:numdenombound},
		\begin{align*}
			\gamma(u) &\geq \gamma(c_2 r' \log r')\\
			&= \tau c_2 r' \log r' - 2(r'+1)(\log c_2 + \log(r' \log r')) + (\log C_2 - \log C_1 - r'\log 2)\\
			&\geq c_2 (\tau r' \log r' - 2(r'+1)) - 2(r'+1)(\log r' + \log\log r') + (\log C_2 - \log C_1 - r'\log 2)\\
			&>0.
		\end{align*}
		Increasing the constant $c_2$ to account for $r'\in\{2,3,\dots,8\}$ gives the desired result. 
	\end{proof}
	
	\begin{proof}[Proof of Theorem \ref{thm:boundCMECm}]
		Let $m=\lceil C_3\cdot (r')^2\cdot \log(r')^2\rceil \leq C_4\cdot (r')^2\cdot \log(r')^2$ with $C_4 = C_3 + 1$. 
		Then
		\[
		\widetilde{a}_m \leq \widetilde{a}(m) < b(m) \leq b_m
		\]
		by Lemma \ref{lem:global-asympt}, Lemma \ref{lem:localvsglobalvialog}, and Lemma \ref{lem:locasymp}, respectively.
	\end{proof}
	
	\begin{theorem}[Theorem \ref{thm:mainPL}]\label{thm:boundCMEC}
		Assume the Tate--Shafarevich Conjecture and Conjecture \ref{conj:non-vani-padicLfct}.
		Then there exists an absolute, positive constant $C_5$ and a constant $\kappa = \kappa(E,p)$ depending only on the prime $p$ and (the number of bad primes of) the curve $E$ such that
		\[
		\# \mathcal{Y}(\bZ_p)_{S,U'} \leq \kappa e^{C_{5}(r+s)^3\log(r+s)^3}
		\]
		whenever $r+s\geq 2$.
	\end{theorem}
	
	\begin{proof}
		We apply \cite[Theorem 6.2.1]{alex:effective} using the bound $m\leq m_0 \coloneqq C_4 (r+s)^2 \log(r+s)^2$ from Theorem \ref{thm:boundCMECm}. For a prime $\ell$, let $n_\ell$ be the number of irreducible components of $(\mathcal{Y}_{\min})_{\mathbb{F}_\ell}$ of multiplicity one.
		Note that in the bound of \cite[Theorem 6.2.1]{alex:effective} we can use this smaller value of $n_\ell$ as a $\bZ_S$-point of $\mathcal{Y}_{\min}$ will always reduce to one of these components.
		Then $n_\ell$ is less than or equal to the number of irreducible components of $(\mathcal{E}_{\min})_{\overline{\mathbb{F}}_\ell}$ of multiplicity one.
		The latter can be read off from the classification of the special fibre of the minimal regular model of elliptic curves, see \cite[p.\ 365]{SilvermanArithEll2}.
		By \cite[Theorem II.6.4]{SilvermanArithEll2}, CM elliptic curves have potentially good reduction, thus cannot have multiplicative reduction.
		Thus we see that $n_\ell \leq 4$ for all primes $\ell$.
		
		For $i\geq 1$, the coefficient $(c'_i)^\loc$ of $\HS'_\loc(t)$ satisfies
		\[
		(c'_i)^\loc + 1 \leq B_i \leq b_i \leq (i+1)p(i)^2 \leq 2 C_1^2 \frac{1}{i} e^{2\pi\sqrt{2i/3}},
		\]
		where the last inequality follows from \eqref{eq:asympt-partition}.
		Now \cite[Theorem 6.2.1]{alex:effective} yields
		\[
		\# \mathcal{Y}_{\min}(\bZ_p)_{S,U'} \leq \kappa_p (\#E(\bF_p)-1) 5^t 4^{m_0} (2C_1^2)^{m_0} \frac{1}{m_0!} e^{2\pi\sqrt{2/3}\sum_{i=1}^{m_0} \sqrt{i}},
		\]
		where $t$ denotes the number of bad primes of $E$.
		Of course $\frac{1}{m_0!}\leq 1$.
		By the summation formula for square roots, e.\,g.\ using the Euler--MacLaurin formula \cite[Theorem 12.27]{Discretecalc}, for all $m$ we have $\sum_{i=1}^m \sqrt{i} \leq \frac{2}{3} m^{3/2} + \frac{1}{2} m^{1/2}$. 
		With $\kappa=\kappa(E,p)\coloneqq\kappa_p (\#E(\bF_p)-1) 5^t$, we get 
		\[
		\# \mathcal{Y}_{\min}(\bZ_p)_{S,U'} \leq \kappa \exp\left\{m_0 \log(8C_1^2) + 2\pi\sqrt{\frac{2}{3}}\left(\frac{2}{3} m_0^{3/2} + \frac{1}{2} m_0^{1/2}\right)\right\}.
		\]
		Finally, the exponent is dominated by a term of the form $c_5 m_0^{3/2}$, where $c_5$ is an absolute constant, and the assertion follows by plugging in $m_0$ and putting $C_{5} = c_5 C_4^{3/2}$. 
	\end{proof}
	
	\begin{remark}\label{rem:boundscomparison}
		\begin{enumerate}
			\item Viewing $r$ as a constant depending on $E$, we see that for a given CM elliptic curve $E$ we have
			\[
			\# \mathcal{Y}_{\min}(\bZ_p)_{S,U'} \leq \kappa e^{C_{6}s^3\log(s)^3},
			\]
			with $C_{6}$ a constant depending on $E$.
			Thus fixing $E$, we get a bound on the $S$-integral points that depends only on $\# S = s$ (and $p$).
			\item We remark again that $\Y(\bZ_S)=\Y_{\min}(\bZ_S) \subset \Y_{\min}(\bZ_p)_{S,U'}$, so we get the same bound for the number of $S$-integral points on the affine Weierstra\ss ~model $\Y$.
			The $j$-invariant of a CM elliptic curve is integral \cite[Theorem II.6.1]{SilvermanArithEll2}, thus \cite[Theorem A]{SilvermanQuantiSiegel} says that there exists a constant $\kappa'$ (independent of $E$ and $p$) such that
			\[
			\# \mathcal{Y}(\bZ_S) \leq (\kappa')^{1+r+s}.
			\]
			For $r+s$ large, this bound is better than the bound of Theorem \ref{thm:boundCMEC}.
		\end{enumerate}
	\end{remark}
	
	\begin{remark}\label{rem:constantcomput}
		In practice, one can compute the constant $\kappa$.
		We determined explicit values for all other appearing constants, starting from an improved version of \eqref{eq:asympt-partition} implied by Lehmer's error bound, see the first inequality in the proof of Theorem 2.1 in \cite{bessenrodtono_partition}.
		Namely, one has $(C_0,C_1)=(\frac{\sqrt{3}}{12}e^{\sqrt{23}-\sqrt{24}},\frac{\sqrt{3}}{6})\approx(0.13019,0.28868)$, leading to $C_2=0.00552$, $c_1=c_2=c_3=147$, $C_3=21609$, $C_4=21610$, $c_5=6$ and finally $C_5=19060462$.
		We do not claim these numbers to be optimal.
	\end{remark}
	
	\begin{remark}\label{rem:boundCMECuncond}
		Let us also present an unconditional result, i.\,e.\ without assuming the Tate--Shafarevich Conjecture nor Conjecture \ref{conj:non-vani-padicLfct}.
		Namely, using Remark \ref{rem:globSeldimonelargek}(1) we know that $\dim \HH^1_f(G_\bQ;V_k')=1$ for $k\gg 0$.
		The statement of Theorem \ref{thm:boundCMECm} holds verbatim with $r'=r+s$ replaced by $r'=r_p+s+\delta$ where $r_p$ is the $p^\infty$-Selmer rank of~$E$ and
		\[
		\delta = \sum_{k\geq3}\bigl(\dim \HH^1_f(G_\bQ;V_k')-1\bigr) \,.
		\]
		The same is true for Theorem \ref{thm:boundCMEC}, proving Theorem \ref{thm:mainPL2} with $C=C_5$ and $r_1 = r_p+\delta$.
	\end{remark}
	
	\begin{remark}
		As M.\ Lüdtke pointed out, we can apply the techniques of this appendix to bound the number of $S$-integral points of $\mathcal{Y}=\mathbb{P}^1_\bZ\setminus\{0,1,\infty\}$ using the polylogarithmic quotient $U'$ of the full $\bQ_p$-pro-unipotent fundamental group $U$ of $Y=\mathbb{P}^1_\bQ\setminus\{0,1,\infty\}$.
		Here our results are unconditional, i.e. do not rely any Bloch--Kato type conjecture like Conjecture \ref{conj:BK} or Conjecture \ref{conj:non-vani-padicLfct}.
		The polylogarithmic quotient is defined in \cite[(16.11)]{Deligne89} (denoted $U$ there).
		Namely, in this case one has 
		\[
		V_k' = \gr^{\W}_{-k} U' = \begin{cases}
			0, & k \text{ odd},\\
			\bQ_p(1)^{\oplus 2}, & k=2,\\
			\bQ_p(k/2), & k\geq 4 \text{ even}.
		\end{cases}
		\]
		(Bounds on) the local and global Selmer dimensions of $\bQ_p(n)$ are known, using \cite[Example 3.9]{BlochKato} and Soul\'e's vanishing theorem \cite{soulevani}, see also \cite[\S 4.3.1]{BellaicheNotes}, leading to the following Hilbert series
		\begin{align*}
			\HS'_\loc(t) &= (1-t^2)^{-2} \prod_{k\geq 3 \text{ even}} (1-t^k)^{-1} = (1-t^2)^{-1} \sum_{n\geq 0} p(n) t^{2n},\\
			\HS'_\glob(t) &\leq (1-t^2)^{-s} \prod_{k\geq 3, ~k\equiv 2 \pmod 4} (1-t^k)^{-1} = (1-t^2)^{-(s-1)} \sum_{n\geq 0} q(n) t^{2n},
		\end{align*}
		where by $\leq$ we mean coefficientwise inequality and $q(n)$ denotes the number of partitions of $n$ into odd parts.
		Repeating the analysis of Theorem \ref{thm:boundCMECm} using \eqref{eq:asympt-partition} and the analogous asymptotic of $q(n)$, see \cite[\S 9.6]{BrownIntegral}, yields a constant $\gamma$ such that the $m$-th coefficient of  $\frac{1}{1-t}\HS'_\glob(t)$ is less than the $m$-th coefficient of $\frac{1}{1-t}\HS'_\loc(t)$
		for some $m\leq \gamma s^2 \log(s)^2$.
		\cite[Theorem 6.2.1(B)]{alex:effective} together with \cite[Lemma 7.0.5]{alex:effective} then imply that there is a constant $\gamma'$ such that
		\[
		\#\mathcal{Y}(\bZ_p)_{S,U'} \leq e^{\gamma' s^2 \log(s)^2}.
		\]
		As in Remark \ref{rem:constantcomput}, it should be possible to make the constants $\gamma$ and $\gamma'$ explicit.
		As was the case for CM elliptic curves, the bound for the quotient $U'$ is better than the bound in \cite[Theorem D]{alex:effective} obtained using the full fundamental group $U$ instead of $U'$.
		However, the bound is still worse than the bound of $3\cdot 7^{2s+1}$ of Evertse (\cite[Theorem 1]{Evertse}).
	\end{remark}
	
	\bibliography{david_references}
	\bibliographystyle{alpha}

\end{document}